\documentclass[12pt,a4paper]{article}

\usepackage{amsmath}
\usepackage{amsfonts}
\usepackage{amssymb}
\usepackage{color}

\usepackage{amsthm}
\usepackage{graphicx}
\usepackage[T1]{fontenc}     
\usepackage{url}
\usepackage{hyperref}

\newtheorem{prop}{Proposition}[section]
\newtheorem{theorem}{Theorem}[section]
\newtheorem{cor}{Corollary}[section]

\newtheorem{defi}{Definition}[section]
\newtheorem{lemma}{Lemma}[section]
\numberwithin{equation}{section}

\makeatletter \@addtoreset{equation}{section}

\AtBeginDocument{

}

\renewcommand{\thefootnote}{\fnsymbol{footnote}}

\begin{document}

\allowdisplaybreaks

\centering \Large On a stochastic Hardy-Littlewood-Sobolev inequality with application to Strichartz estimates for the white noise dispersion
\vspace{2em}

\large
\begin{tabular}{c c c}
&&\\
Romain Duboscq \footnotemark[1] & & Anthony R\'eveillac \footnotemark[2] \vspace{1em} \normalsize\\
\end{tabular}

\begin{center}
INSA de Toulouse \footnotemark[3] \\
IMT UMR CNRS 5219 \\
Universit\'e de Toulouse 
\end{center}

\footnotetext[1]{\texttt{romain.duboscq@insa-toulouse.fr}}
\footnotetext[2]{\texttt{anthony.reveillac@insa-toulouse.fr}}
\footnotetext[3]{135 avenue de Rangueil 31077 Toulouse Cedex 4 France}

\abstract{\noindent 
In this paper, we investigate a stochastic Hardy-Littlewood-Sobolev inequality. Due to the stochastic nature of the inequality, the relation between the exponents of integrability is modified. This modification can be understood as a regularization by noise phenomenon. As a direct application, we derive Strichartz estimates for the white noise dispersion which enables us to address a conjecture from \cite{belaouar2015numerical}.}

\vspace{1em}
{\noindent \textbf{Key words:} Stochastic regularization; Stochastic Partial Differential Equations; Nonlinear Schrödinger equation, Hardy-Littlewood-Sobolev inequality.
}

\vspace{1em}
\noindent
{\noindent \textbf{AMS 2010 subject classification:} Primary: 35Q55; Secondary: 60H15.}

\renewcommand{\thefootnote}{\fnsymbol{footnote}}

\allowdisplaybreaks

\section{Introduction}

Let $(\Omega, \mathbb{P})$ be the standard probability space endowed with the Wiener filtration $(\mathcal{F}_t)_{t\geq 0}$. The main objective of this paper is to address the local Cauchy problem for the following nonlinear Schrödinger equation with white noise dispersion
\begin{equation}\label{eq:main}
\left\{\begin{array}{ll}
i d\psi(t,x) = -\Delta \psi(t,x) \circ dW_t(\omega) + \lambda |\psi|^{2\sigma} \psi(t,x)dt,\quad \forall (t,x,\omega)\in[0,T]\times\mathbb{R}^d\times\Omega,
\\ \psi(0,\cdot) = \psi_0\in L^2(\mathbb{R}^d),
\end{array}\right.
\end{equation}
where $\lambda \in\mathbb{R}$ and $(W_t)_{t\geq 0}$ is the Wiener process and the product $\circ$ is understood in the Stratonovich sense.

For $d=1$ and $\sigma = 1$, this equation arises in the field of nonlinear optics as a limit model for the propagation of light pulse in an optical fiber where the dispersion varies along the fiber \cite{agrawal2007nonlinear,agrawal2001applications}. These variations in the dispersion accounts for the so-called \textit{dispersion management} which aims to improve the transmission of a light signal by constructing a zero-mean dispersion fiber in order to avoid the problem of the chromatic dispersion of the light signal. When the variations are assumed to be random, a noisy dispersion can be derived  (see \cite{marty2006splitting, de2010nonlinear}) which leads, in the white noise case, to Equation \eqref{eq:main}.

As part of the problems concerning the propagation of waves in random media, there is a vast literature around random Schrödinger equations. Let us mention in particular the cases of random potentials \cite{erdHos2008quantum,erdHos2000linear} and noisy potentials \cite{de1999stochastic,de2002effect,de2005blow}. In these works, the effects of the stochastic potential greatly affect the dynamic of the Schrödinger equation and are, in a broader context, a motivation to introduce randomness in PDEs. Specifically, there is a well known effect which attracted a lot of attention: the so-called \textit{regularization by noise} phenomenon (see \cite{flandoli2011random} for a survey). This phenomenon can be summarized as an improvement, due to the presence of noise, of the well-posedness of differential equations and has been studied in the context of SDEs \cite{zvonkin1974transformation,veretennikov1981strong,krylov2005strong,priola2012pathwise,CATELLIER20162323}, transport equation \cite{flandoli2010well,fedrizzi2013noise,catellier2016rough}, SPDEs \cite{da2013strong} and scalar conservation laws \cite{gess2017long}. We remark that obtaining a regularization by noise in the context of nonlinear random PDE is a challenging task and most of the results are obtained in a linear setting. For instance, an open problem is to obtain a regularization by noise for the Euler or Navier-Stokes equations.

We are not the first one to investigate the Cauchy problem of Equation \eqref{eq:main}. It was first studied in \cite{de2010nonlinear} where the global Cauchy problem was solved for $\sigma < d/2$ which corresponds to a classical $L^2$-subcritical nonlinearity. In \cite{debussche20111d}, the authors proved that, in the $L^2$-critical case, when $d = 1$ and $\sigma = 5$, the solutions are globally well-posed, which is not the case for the deterministic nonlinear Schrödinger equation and, thus, hints for a regularization by noise effect. In \cite{chouk2015nonlinear}, the authors study the case where the Wiener process is replaced by a fractional Wiener process and recover similar results as in \cite{de2010nonlinear}. By a simple scaling argument on the space and time variables of \eqref{eq:main} and thanks to the scaling invariance of the Wiener process, it was conjectured in \cite{belaouar2015numerical} that, in fact, the critical nonlinearity should be $\sigma = 4/d$, a $L^2$-supercritical nonlinearity, which is twice as large as the deterministic $L^2$-critical nonlinearity. Furthermore, this fact was supported by numerical simulations in $1$D and leads to believe that the white noise dispersion has a strong regularization effect.

In this paper, we address the global Cauchy problem \eqref{eq:main} for $\sigma<4/d$. To be more specific, we obtain the following result.

\begin{theorem}\label{thm:main}
Let $\sigma<\frac{4}{d}$, $\psi_0\in L^2(\mathbb{R}^d)$ and $a\in (2,\infty)$ such that $\frac{d\sigma}4< \frac{2(\sigma+1)}a< 1$.
Then, for almost all $\omega\in\Omega$, there exists a unique solution $\psi\in L^{a}([0,+\infty[;L^{2\sigma+2}(\mathbb{R}^d))$ to \eqref{eq:main}.
\end{theorem}

The classical approach to investigate the Cauchy problem for nonlinear Schrödinger equations is to derive local Strichartz estimates \cite{cazenave2003semilinear}. These estimates are a direct consequence of the dispersive property of the linear operator $i\Delta$. However, as pointed out in \cite{debussche20111d}, it is much harder to obtain such estimates in the case of a white noise dispersion because of the presence of the Wiener process. We remark that the strategy used in \cite{de2010nonlinear} does rely on stochastic Strichartz estimates but these were not efficient enough to handle $L^2$-supercritical nonlinearities.

Let us now explain our approach to deduce Strichartz estimates for \eqref{eq:main}. We recall from \cite{marty2006splitting,de2010nonlinear} that the propagator associated to the linear part of \eqref{eq:main} is explicitly given by, $\forall t,s\in (0,\infty)$, $\forall \omega\in \Omega$ and $\forall \varphi\in\mathcal{C}^{\infty}_0(\mathbb{R}^d)$,
\begin{align}
P_{s,t}(\omega)\varphi(x) : &= (2\pi)^{-d}\int_{\mathbb{R}^d} e^{-i|\xi|^2(W_t(\omega)-W_s(\omega))} \hat{\varphi}(\xi) e^{i\xi\cdot x} d\xi \nonumber
\\ &= \frac{1}{(4\pi(W_t(\omega)-W_s(\omega)))^{d/2}}\int_{\mathbb{R}^d} e^{i\frac{|x-y|^2}{4(W_t(\omega)-W_s(\omega))}} \varphi(y)dy.\label{eq:propagator}
\end{align}

Following the classical proof of Strichartz estimates (see for instance \cite{keel1998endpoint}), a fundamental tool is the Hardy-Littlewood-Sobolev inequality \cite{hardy1928some,hardy1932some,sobolev1938theorem} which is stated below. 
\begin{theorem}\label{thm:HLS}
Let $T>0$, $\alpha\in (0,1)$ and $f\in L^p([0,T])$ and $g\in L^q([0,T])$ such that $p,g\in(1,\infty)$ and
\begin{equation*}
2-\alpha = \frac 1 p + \frac 1 q.
\end{equation*}
Then, there exists a constant $C>0$ which depends on $p$ and $q$ such that the following inequality holds
\begin{equation}\label{eq:cHLS}
\left|\int_0^T \int_0^T f(t)|t-s|^{-\alpha} g(s) ds dt \right| \leq C \|f\|_{L^p([0,T])} \|g\|_{L^q([0,T])}.
\end{equation}
\end{theorem}

From here, if we wish to follow the classical arguments to derive dispersive estimates, the main difficulty is to prove inequality \eqref{eq:cHLS} but replacing the potential $|t-s|^{-\alpha}$ with $|W_t(\omega)-W_s(\omega)|^{-\alpha}$. This is the point of the following theorem, which is the second result of this paper.

\begin{theorem}\label{thm:HLSsto}
Let $T>0$, $p,q\in (1,\infty)$ and $\alpha\in (0,1)$ such that
\begin{equation*}
2 - \frac{\alpha}2 > \frac1p+\frac1q.
\end{equation*}
Then, there exists a set $\mathcal{N}\subset \Omega$ of zero measure  which depends on $T$ and $\alpha$ such, $\forall \omega \notin\mathcal{N}$, $\forall f\in L^p([0,T])$, $\forall g\in L^q([0,T])$, the following inequality holds 
\begin{equation}\label{ineq:HLSsto}
\left|\int_0^T\int_0^T f(t)|W_t(\omega)-W_s(\omega)|^{-\alpha}g(s) ds dt\right|\leq \|f\|_{L^p([0,T])}\|g\|_{L^q([0,T])}.
\end{equation}
\end{theorem}

We can see that our result does not give an equality between the exponents of integrability and $\alpha$ but an inequality. However, it is enough for our purpose. We also remark that the relation is very different from the one in Theorem \ref{thm:HLS} since $\alpha$ is divided by a factor of $2$. This is due to the stochastic nature of the potential $|W_t-W_s|^{-\alpha}$ and, somehow, is a consequence of the scaling invariance of the Wiener process. This modification has a dramatic impact on the integrability assumptions of $f$ and $g$ and, as a consequence, we obtain the following stochastic Strichartz estimates for the propagator $(P_{s,t})_{s,t\geq 0}$ given by \eqref{eq:propagator}.

\begin{defi}
For any $(q,p)\in (1,\infty)^2$ we say that $(q,p)$ is sub-admissible if
\begin{equation*}
\frac{2}{q} > \frac{d}{2}\left(\frac 1 2 - \frac 1 p \right).
\end{equation*}
\end{defi}
\begin{prop}\label{prop:Strichartz}
Let $T>0$ and $(q,p)$ sub-admissible. Then, there exists two constants $C_1,C_2>0$ and a set $\mathcal{N}\subset \Omega$ of zero measure which depends on $d$, $T$, $p$ and $q$ such that, $\forall \omega \notin \mathcal{N}$, $\forall f\in L^2(\mathbb{R}^d)$ and $\forall g\in L^{r'}([0,T];L^{l'}(\mathbb{R}^d))$, the following inequalities holds 
\begin{align}
\left\|P_{0,\cdot}(\omega) f \right\|_{L^q([0,T];L^p(\mathbb{R}^d))}&\leq  C_1\|f\|_{L^2}\label{eq:Stri1},
\\ \left\|\int_0^T P_{s,\cdot}(\omega) g(s)ds \right\|_{L^q([0,T];L^p(\mathbb{R}^d))} &\leq C_2 \|g\|_{L^{r'}([0,T];L^{l'}(\mathbb{R}^d))},\label{eq:Stri2}
\end{align}
for any $(r,l)$ sub-admissible.
\end{prop}

Thanks to the previous result, we are able to prove Theorem \ref{thm:main} by classical arguments. The rest of the paper is devoted to the proof of Theorem \ref{thm:HLSsto} in section \ref{sec:proofThmHLSsto} and the proofs of Proposition \ref{prop:Strichartz} and Theorem \ref{thm:main} in section \ref{sec:proofThmMain}.

\section{Proof of Theorem \ref{thm:HLSsto}}\label{sec:proofThmHLSsto}

Let $T>0$, $p,q\in (1,\infty)$, $\alpha\in (0,1)$ and $\upsilon>0$ such that
\begin{equation*}
2 - \frac{\alpha}2 -\upsilon =  \frac1p+\frac1q.
\end{equation*}
Before proceeding any further, let us remark that we can, without loss of generality, assume that $f\in L^p([0,T])$ and $g\in L^q([0,T])$ are non-negative functions. 

Let $N\in\mathbb{N}^*$ and $f^{(N)},g^{(N)}$ be two simple function given by\begin{equation*}
f^{(N)}(t) = \sum_{j = 1}^{2^N-1} f_j \bold{1}_{[t_j,t_{j+1})}(t)\quad\textrm{and}\quad g^{(N)}(s) = \sum_{k = 1}^{2^N-1} g_k \bold{1}_{[t_k,t_{k+1})}(s),
\end{equation*}
where $(f_j)_{1\leq j\leq 2^N-1},(g_k)_{1\leq k\leq 2^N-1}\subset\mathbb{R}^+$ and $(t_j)_{1\leq j\leq 2^N-1}$ is a uniform discretization of $[0,T]$ such that
\begin{equation*}
t_j : =  j h_N, \quad \forall j\in\{0,\cdots,2^N\},
\end{equation*}
with
\begin{equation*}
h_n :=\frac{T}{2^N}.
\end{equation*}
We remark that
\begin{equation*}
\|f^{(N)}\|_{L^p} = \left(\sum_{j = 1}^{2^N-1} f_j^p h \right)^{1/p}\quad\textrm{and}\quad\|g^{(N)}\|_{L^q} = \left(\sum_{k = 1}^{2^N-1} g_k^q h \right)^{1/q}.
\end{equation*}
Furthermore, it follows that
\begin{align*}
\int_0^T\int_0^T f^{(N)}(t)|W_t(\omega)-W_s(\omega)|^{-\alpha}g^{(N)}(s) ds dt &=
\\ &\hspace{-8em} \sum_{j,k = 1}^{2^N-1}f_jg_k \int_{t_j}^{t_{j+1}}\int_{t_k}^{t_{k+1}}|W_t(\omega)-W_s(\omega)|^{-\alpha}dsdt.
\end{align*}
From here, we need the following proposition.
\begin{prop}\label{prop:Sets}
There exists a sequence of sets $\left\{\Omega_{\varepsilon}\right\}_{\varepsilon>0}$ which depends on $T$ and $\alpha$ such that 
\begin{enumerate}
\item $\Omega_{\varepsilon_1}\subset \Omega_{\varepsilon_2}$ for all $\varepsilon_1<\varepsilon_2$,
\item $\mathbb{P}(\Omega_{\varepsilon}) \geq 1-\varepsilon$ for all $\varepsilon>0$,
\item there exists a $N_{\varepsilon}$ large enough such that, $\forall \nu>0$, $\forall \omega\in \Omega_{\varepsilon}$, 
\begin{equation}
\int_{t_j}^{t_{j+1}}\int_{t_k}^{t_{k+1}}|W_t(\omega)-W_s(\omega)|^{-\alpha}dsdt \leq h_{N_{\varepsilon}}^{2-\alpha/2 - \nu}.
\end{equation}
\end{enumerate}

\end{prop}
Assume for a moment that Proposition \ref{prop:Sets} holds. Then, for any $N\geq N_{\varepsilon}$, we can deduce that, thanks to Jensen's inequality and since $2-\alpha/2 - \upsilon = 1/p+1/q$,
\begin{align*}
\int_0^T\int_0^T f^{(N)}(t)|W_t(\omega)-W_s(\omega)|^{-\alpha}g^{(N)}(s) ds dt&
 \\&\hspace{-12em}\leq \sum_{j,k = 1}^{2^N-1}f_jg_k h_N^{\frac1p+\frac1q}
\\&\hspace{-12em}\leq \left( \sum_{j= 1}^{2^N-1} f_j h_N^{\frac1p} \right)\left( \sum_{k= 1}^{2^N-1} g_k h_N^{\frac1q} \right)
\\&\hspace{-12em}\leq  \left( \sum_{j= 1}^{2^N-1} f_j^p h_N \right)^{\frac1p}\left( \sum_{k= 1}^{2^N-1} g_k^q h_N \right)^{\frac1q} = \|f^{(N)}\|_{L^p(\mathbb{R}^d)}\|g^{(N)}\|_{L^q(\mathbb{R}^d)},
\end{align*}
which is exactly \eqref{ineq:HLSsto}. By a density argument and Fatou's lemma, we obtain that, $\forall \varepsilon>0$, there exists a set $\Omega_{\varepsilon}$ such that, $\forall \omega\in \Omega_{\varepsilon}$, $\forall f\in L^p([0,T])$, $\forall g\in L^q([0,T])$, the following inequality holds
\begin{equation}\label{ineq:HLSstopr}
\int_0^T\int_0^T f(t)|W_t(\omega)-W_s(\omega)|^{-\alpha}g(s) ds dt\leq \|f\|_{L^p([0,T])}\|g\|_{L^q([0,T])}.
\end{equation}
Since this estimate is uniform in $\varepsilon$, we can pass to the limit $\varepsilon\to 0$. Set $\varepsilon = 1/n$, we have
\begin{equation*}
\mathbb{P}\left[\cap_{n = 1}^{+\infty}\Omega_{1/n}\right] = \lim_{n\to+\infty}\mathbb{P}[\Omega_{1/n}] \geq 1.
\end{equation*}
Hence, there exists a set $\mathcal{N}$ of zero measure, which depends on $\alpha$ and $T$, such that, $\forall \omega\notin \mathcal{N}$, $\forall f \in L^p([0,T])$ and $\forall g\in L^q([0,T])$, the estimate \eqref{ineq:HLSstopr} holds. Thus, to conclude the proof of Theorem \ref{thm:HLSsto}, it remains to prove Proposition \ref{prop:Sets}. Before proceeding further, we need some technical results.

Let us begin with the following estimate.
\begin{lemma}\label{lem:Moments}
Let $p\geq 1$ and $j,k\in\mathbb{N}$. There exists a constant $C>0$ such that the following estimate holds
\begin{equation*}
\mathbb{E}\left[ \left( \int_{t_j}^{t_{j+1}}\int_{t_k}^{t_{k+1}} |W_t-W_s|^{-\alpha}dsdt\right)^p \right] \lesssim C^{p} p(p!) h_N^{p(2-\alpha/2)}.
\end{equation*}
\end{lemma}
\begin{proof}
We first remark that, since $t_j = jh$ and by the scaling property of the Brownian motion,
\begin{align*}
\mathbb{E}\left[ \left( \int_{t_j}^{t_{j+1}}\int_{t_k}^{t_{k+1}} |W_t-W_s|^{-\alpha}dsdt\right)^p \right] &
\\ &\hspace{-8em} = h^{p(2-\alpha/2)}\mathbb{E}\left[ \left( \int_{j}^{j+1}\int_{k}^{k+1} |W_t-W_s|^{-\alpha}dsdt\right)^p \right].
\end{align*}
We now separate the proof in two parts: the first part where $j>k$, which, by symmetry, also gives the case $j>k$, and the second part when $j=k$.

\noindent\textbf{Step 1: the case $j>k$.} By denoting $\mathbf{t} = (t_1,t_2,\cdots,t_p)\in\mathbb{R}^p$ and $\mathbf{s} = (s_1,s_2,\cdots,s_p)\in\mathbb{R}^p$, we remark that
\begin{align*}
\mathbb{E}\left[ \left(  \int_{j}^{j+1}\int_{k}^{k+1} |W_t-W_s|^{-\alpha}dsdt\right)^p \right] &= \int_{[j,j+1]^p} \mathbb{E}\left[ \prod_{n = 1}^p \int_{k}^{k+1}  |W_{t_n}-W_{s}|^{-\alpha}ds\right]d\mathbf{t} 
\\ &\hspace{-9em}= (p!)\int_{\boldsymbol{\Delta}_p([j,j+1])}\mathbb{E}\left[\prod_{n = 1}^p \int_{k}^{k+1}  |W_{t_n}-W_{s}|^{-\alpha}ds\right]d\mathbf{t}
\\ &\hspace{-9em}= (p!)\int_{\boldsymbol{\Delta}_p([j,j+1])} \int_{[k,k+1]^p} \mathbb{E}\left[\prod_{n = 1}^p |W_{t_n}-W_{s_n}|^{-\alpha}\right]d\mathbf{s}d\mathbf{t}
\\ &\hspace{-9em}= (p!)\sum_{\sigma\in\mathfrak{G}_p}\int_{\boldsymbol{\Delta}_p([j,j+1])} \int_{\boldsymbol{\Delta}_p([k,k+1])} \mathbb{E}\left[\prod_{n = 1}^p |W_{t_n}-W_{s_{\sigma(n)}}|^{-\alpha}\right]d\mathbf{s}d\mathbf{t},
\end{align*}
where we used the fact that, for all $(s_1,s_2,\cdots,s_p)\in\mathbb{R}^p$,
\begin{equation*}
\sum_{\sigma\in\mathfrak{G}_p} \bold{1}_{s_{\sigma(1)}\leq s_{\sigma(2)}\leq \cdots\leq s_{\sigma(p)}} = 1,
\end{equation*}
with $\mathfrak{S}_p$ is the set of permutations of length $p$ and $\boldsymbol{\Delta}_{p}([\alpha,\beta]) = \{\mathbf{t}\in\mathbb{R}^p; \alpha<t_1<t_2<\cdots<t_p<\beta\}$. 

Since $k+1\leq j$, we have $s_p<t_1$ and, hence,
\begin{equation*}
s_1<s_2<\cdots<s_p<t_1<t_2<\cdots<t_p.
\end{equation*}
By the Markov property of the Brownian motion and denoting the conditional expectation, $\forall t\geq 0$, $\forall X\in L^1(\Omega)$,
\begin{equation*}
\mathbb{E}_t\left[X\right] := \mathbb{E}\left[X\middle| \mathcal{F}_t\right],
\end{equation*}
we obtain that
\begin{align*}
\mathbb{E}_{t_{p-1}}\left[\prod_{n = 1}^p |W_{t_n}-W_{s_{\sigma(n)}}|^{-\alpha}\right] &= \prod_{n = 1}^{p-1} |W_{t_n}-W_{s_{\sigma(n)}}|^{-\alpha} \mathbb{E}_{t_{p-1}}\left[ |W_{t_p}-W_{s_{\sigma(p)}}|^{-\alpha}\right]
\\ &\hspace{-10em}=  \prod_{n = 1}^{p-1}  |W_{t_n}-W_{s_\sigma(n)}|^{-\alpha}\int_{\mathbb{R}} |x_p + W_{t_{p-1}}-W_{s_\sigma(p)}|^{-\alpha} G_{t_p - t_{p-1}}(x_p) dx_p,
\end{align*}
where $(G_t)_{t\geq 0}$ is the Gaussian kernel. We remark that the following estimate holds
\begin{equation}\label{eq:estGaussPot}
\sup_{y\in\mathbb{R}}\int_{\mathbb{R}} |x-y|^{-\alpha} G_1(x) dx \leq C,
\end{equation}
for a certain constant $C>0$. It then follows from the scaling property of the Gaussian kernel, by denoting $\delta t_j = t_j-t_{j-1}$ and the estimate \eqref{eq:estGaussPot} that
\begin{align*}
\mathbb{E}_{t_{p-1}}\left[\prod_{n = 1}^p |W_{t_n}-W_{s_{\sigma(n)}}|^{-\alpha}\right] &
\\ &\hspace{-10em}= ( \delta t_p)^{-\alpha/2}\prod_{n = 1}^{p-1}  |W_{t_n}-W_{s_\sigma(n)}|^{-\alpha}\int_{\mathbb{R}} |x + (\delta t_p)^{-1}(W_{t_{p-1}}-W_{s_\sigma(p)})|^{-\alpha} G_{1}(x) dx
\\ &\hspace{-10em}\leq C ( \delta t_p)^{-\alpha/2}\prod_{n = 1}^{p-1}  |W_{t_n}-W_{s_\sigma(n)}|^{-\alpha}.
\end{align*}
By the tower property of the conditional expectations and an induction argument, we deduce that
\begin{align*}
\mathbb{E}\left[\prod_{n = 1}^p |W_{t_n}-W_{s_{\sigma(n)}}|^{-\alpha}\right] &\leq C^{p-1} \prod_{n = 2}^p  ( \delta t_n)^{-\alpha/2} \mathbb{E}\left[|W_{t_1} - W_{s_{\sigma(1)}}|^{-\alpha}\right]
\\ &\leq C^p (t_1-s_{\sigma(1)})^{-\alpha/2} \prod_{n = 2}^p  ( \delta t_n)^{-\alpha/2}.
\end{align*}
Thus, we have, since 
\begin{align*}
\mathbb{E}\left[ \left(  \int_{j}^{j+1}\int_{k}^{k+1} |W_t-W_s|^{-\alpha}dsdt\right)^p \right]&
\\ &\hspace{-12em}\leq  C^p (p!)\sum_{\sigma\in\mathfrak{S}_p} \int_{\boldsymbol{\Delta}_p([j,j+1])} \int_{\boldsymbol{\Delta}_p([k,k+1])}(t_1-s_{\sigma(1)})^{-\alpha/2} \prod_{n = 2}^p  ( \delta t_n)^{-\alpha/2}d\mathbf{s}d\mathbf{t}.
\end{align*}
We remark that, if $j = k+1$ and since $t_1\in [k+1,k+2]$,
\begin{align*}
\int_{\boldsymbol{\Delta}_p([k,k+1])}(t_1-s_{\sigma(1)})^{-\alpha/2}d\mathbf{s} &\leq \int_{\boldsymbol{\Delta}_{p-1}([k,k+1])} d\mathbf{s} \int_{k}^{k+1}(t_1-s)^{-\alpha/2} ds
\\ &\leq \frac{1}{(p-1)!}\sup_{t\in[k+1,k+2]}(t_1-k)^{1-\alpha/2} \lesssim \frac{1}{(p-1)!}.
\end{align*}
Else, if $j\geq k+2$, we have
\begin{align*}
\int_{\boldsymbol{\Delta}_p([k,k+1])}(t_1-s_{\sigma(1)})^{-\alpha/2}d\mathbf{s} &\leq\sup_{t\in[j,j+1],s\in[k,k+1]}(t-s)^{-\alpha/2}  \int_{\boldsymbol{\Delta}_{p}([k,k+1])} d\mathbf{s} 
\\ &\lesssim \frac{1}{p!}.
\end{align*}
In order to estimate the term
\begin{align*}
\int_{\boldsymbol{\Delta}_p([j,j+1])}  \prod_{n = 2}^p  ( \delta t_n)^{-\alpha/2}d\mathbf{t},
\end{align*}
we proceed by induction thanks to the following estimate, for all $2\leq n\leq p$,
\begin{align*}
\int_{k}^{t_n}(\delta t_n)^{-\alpha/2}dt_{n-1} = \int_{k}^{t_n}(t_n - t_{n-1})^{-\alpha/2}dt_{n-1} = \frac{1}{1-\alpha/2}\left(t_n - k\right)^{1-\alpha/2}\leq 1.
\end{align*}
This leads us to the following bound 
\begin{equation*}
\int_{\boldsymbol{\Delta}_p([j,j+1])}  \prod_{n = 2}^p  ( \delta t_n)^{-\alpha/2}d\mathbf{t} \leq 1.
\end{equation*}
Finally, since $\textrm{Card}(\mathfrak{S}_p) = p!$, we obtain the desired estimate
\begin{align*}
\mathbb{E}\left[ \left(  \int_{j}^{j+1}\int_{k}^{k+1} |W_t-W_s|^{-\alpha}dsdt\right)^p \right]&
\\ &\hspace{-10em}\leq \sum_{\sigma\in\mathfrak{G}_p} C^p \max(1, p) \leq C^p p (p!).
\end{align*}

\noindent\textbf{Step 2: the cas $j=k$}  By denoting $\mathbf{t} = (t_1,t_2,\cdots,t_{2p})\in\mathbb{R}^{2p}$, we remark that
\begin{align*}
\mathbb{E}\left[ \left(  \int_{j}^{j+1}\int_{j}^{j+1} |W_t-W_s|^{-\alpha}dsdt\right)^p \right] &= \mathbb{E}\left[ \int_{[j,j+1]^{2p}}  \prod_{n = 1}^p |W_{t_{n+p}}-W_{t_{n}}|^{-\alpha}d\mathbf{t} \right] 
\\ &\hspace{-6em}= \sum_{\sigma\in\mathfrak{S}_{2p}} \int_{\boldsymbol{\Delta}_{2p}([j,j+1])} \mathbb{E}\left[\prod_{n = 1}^p |W_{t_{\sigma(n+p)}}-W_{t_{\sigma(n)}}|^{-\alpha}\right]d\mathbf{t}.
\end{align*}
Let $k = \textrm{argmax}_{1\leq \ell \leq 2p} \sigma(\ell)$. Then, we have
\begin{align*}
\mathbb{E}_{t_{2p-1}}\left[\prod_{n = 1}^p |W_{t_{\sigma(n)}}-W_{t_{\sigma(n+p)}}|^{-\alpha}\right] &
\\ &\hspace{-8em}= \prod_{\substack{n = 1\\ n\neq k}}^p |W_{t_{\sigma(n)}}-W_{t_{\sigma(n+p)}}|^{-\alpha}\mathbb{E}_{t_{2p-1}}\left[ |W_{t_{2p}}-W_{t_{\sigma(k+p)}}|^{-\alpha}\right]
\end{align*}
We have, thanks to scaling property of the Gaussian kernel and estimate \eqref{eq:estGaussPot},
\begin{align*}
\mathbb{E}_{t_{2p-1}}\left[ |W_{t_{2p}}-W_{t_{\sigma(k+p)}}|^{-\alpha}\right] &= \int_{\mathbb{R}} \left|x + W_{t_{2p-1}} -  W_{t_{\sigma(k+p)}}\right|^{-\alpha}G_{t_{2p}-t_{2p-1}}(x) dx
\\ &= (\delta t_{2p})^{-\alpha/2} \int_{\mathbb{R}} \left|x + (\delta t_{2p})^{-1}(W_{t_{2p-1}} -  W_{t_{\sigma(k+p)}})\right|^{-\alpha}G_{1}(x) dx
\\ &\leq C(\delta t_{2p})^{-\alpha/2}.
\end{align*}
By repeating this procedure, we obtain a $p$-tuple $\boldsymbol{k} = (k_1,k_2,\cdots, k_p)\in \{1,\cdots,2p\}^{p}$ such that, by integrating out the $p$ singularities in time,
\begin{align*}
\int_{\boldsymbol{\Delta}_{2p}([j,j+1])} \mathbb{E}\left[\prod_{n = 1}^p |W_{t_{\sigma(n+p)}}-W_{t_{\sigma(n)}}|^{-\alpha}\right]d\mathbf{t} &\leq C^p \int_{\boldsymbol{\Delta}_{2p}([j,j+1])}\prod_{n \in \boldsymbol{k}} (\delta t_n)^{-\alpha/2}d\mathbf{t}
\\ &\lesssim \frac{C^p}{p!}
\end{align*}
We deduce that
\begin{align*}
\mathbb{E}\left[ \left(  \int_{j}^{j+1}\int_{j}^{j+1} |W_t-W_s|^{-\alpha}dsdt\right)^p \right] &\lesssim \sum_{\sigma\in\mathfrak{S}_{2p}} \frac{C^p}{p!}
\\ &\lesssim C^p \frac{(2p)!}{p!}.
\end{align*}
It follows from Stirling approximation that
\begin{align*}
\frac{(2p)!}{p!} \lesssim \frac{(2p)^{2p + 1/2}e^{-2p}}{p^{p+1/2}e^{-p}}  =2^{1/2} 4^p p^pe^{-p} \simeq \frac{4^p}{\sqrt{p}} p! \leq p 4^p p! ,
\end{align*}
which gives the desired estimate.
\end{proof}

The previous result enables us to deduce the following Lemma.
\begin{cor}\label{cor:Prob}
We have the following limit
\begin{equation*}
\lim_{\kappa\to+\infty}\mathbb{P}\left[ \left\{ \int_{t_j}^{t_{j+1}}\int_{t_k}^{t_{k+1}} |W_t-W_s|^{-\alpha}dsdt > \kappa h_N^{2-\alpha/2}, \; 1\leq j,k \leq 2^N-1,\; \forall N\geq 1\right\}\right] = 0.\end{equation*}
\end{cor}

\begin{proof}
Let $\kappa> 4C\log(2) $ where $C$ is the constant from Lemma \ref{lem:Moments}. Denote $\theta = h_N^{\alpha-2}/(2C)$, we have,thanks to Chebyshev's inequality, for all $1\leq j,k\leq 2^N-1$,
\begin{align*}
\mathbb{P}\left[  \int_{t_j}^{t_{j+1}}\int_{t_k}^{t_{k+1}} |W_t-W_s|^{-\alpha}dsdt > \kappa N h_N^{2-\alpha/2}\right] &= \mathbb{P}\left[  e^{\theta \int_{t_j}^{t_{j+1}}\int_{t_k}^{t_{k+1}} |W_t-W_s|^{-\alpha}dsdt } > e^{N\frac{\kappa}{2C}}\right]
\\ &\leq  e^{-N\frac{\kappa}{2C}} \mathbb{E}\left[e^{\theta\int_{t_j}^{t_{j+1}}\int_{t_k}^{t_{k+1}} |W_t-W_s|^{-\alpha}dsdt }\right].
\end{align*}
It then follows from Lemma \ref{lem:Moments} that
\begin{align*}
 \mathbb{E}\left[e^{\theta\int_{t_j}^{t_{j+1}}\int_{t_k}^{t_{k+1}} |W_t-W_s|^{-\alpha}dsdt }\right] &= \sum_{p = 1}^{+\infty} \frac{\theta^p}{p!} \mathbb{E}\left[\left(\int_{t_j}^{t_{j+1}}\int_{t_k}^{t_{k+1}} |W_t-W_s|^{-\alpha}dsdt\right)^p\right]
\\ &\lesssim \sum_{p =0 }^{+\infty} p \theta^p C^p h_N^{p(2-\alpha/2)} =  \sum_{p =0 }^{+\infty} p 2^{-p} \lesssim 1.  
\end{align*}
Thus, we have
\begin{align*}
\mathbb{P}\left[  \int_{t_j}^{t_{j+1}}\int_{t_k}^{t_{k+1}} |W_t-W_s|^{-\alpha}dsdt > \kappa N h_N^{2-\alpha/2}\right] \lesssim  e^{-N\frac{\kappa}{2C}},
\end{align*}
which leads to
\begin{align*}
\mathbb{P}\left[ \left\{ \int_{t_j}^{t_{j+1}}\int_{t_k}^{t_{k+1}} |W_t-W_s|^{-\alpha}dsdt > \kappa N h_N^{2-\alpha/2}, \; 1\leq j,k\leq 2^N-1,\; \forall N\geq 1\right\}\right]&
\\ &\hspace{-30em} \lesssim \sum_{N = 1}^{+\infty} \sum_{j,k = 1}^{2^N-1} e^{-N\frac{\kappa}{2C}} =\sum_{N = 1}^{+\infty} 2^{2N-2}e^{-N\frac{\kappa}{2C}} \simeq \sum_{N = 1}^{+\infty} e^{-N(\frac{\kappa}{2C}-2\log(2))} = : I(\kappa).
\end{align*}
Hence, since $\kappa> 4C\log(2) $, we have that $I(\kappa)<+\infty$. Furthermore, by theorem of dominated convergence, we deduce
\begin{equation*}
I(\kappa) \underset{\kappa\to+\infty}{\to}0.
\end{equation*} \end{proof}

We can now proceed to prove Proposition \ref{prop:Sets}. By denoting
\begin{equation*}
\Omega_{\varepsilon} : = \left\{ \int_{t_j}^{t_{j+1}}\int_{t_k}^{t_{k+1}} |W_t-W_s|^{-\alpha}dsdt \leq \kappa_{\varepsilon} N h_N^{2-\alpha/2}, \; 1\leq j,k \leq 2^N-1\; \forall N\geq N_0\right\},
\end{equation*}
we deduce from Corollary \ref{cor:Prob} that, $\forall \varepsilon>0$, there exist $\kappa_{\varepsilon}>0$ such that
\begin{equation*}
\mathbb{P}(\Omega_{\varepsilon}) \geq 1- \varepsilon.
\end{equation*}
Furthermore, we see that $\kappa_{\varepsilon}$ is non-increasing with respect to $\varepsilon$ and, thus, we deduce that the sequence $\{\Omega_{\varepsilon}\}_{\varepsilon>0}$ is increasing, \textit{i.e.} $\Omega_{\varepsilon_1}\subset \Omega_{\varepsilon_2}$, for $\varepsilon_1<\varepsilon_2$. We finally remark that $\Omega_{\varepsilon}$ depends on $\alpha$ and $T$.

\section{Proof of Proposition \ref{prop:Strichartz} and Theorem \ref{thm:main}}\label{sec:proofThmMain}

We can now proceed to prove the Strichartz estimates by the $TT^*$ strategy (see \cite{keel1998endpoint}) and use them in a fixed-point argument  (see \cite{cazenave2003semilinear}) to prove the global well-posedness of Equation \eqref{eq:main}. 

\subsection{Proof of Proposition \ref{prop:Strichartz}}

We deduce, since $P_{s,t}(\omega)$ is an isometry from $L^2$ to itself, thanks to the Hausdorff-Young inequality and an interpolation argument, that, $\forall p\in [2,\infty]$, $\forall \varphi\in L^{p'}(\mathbb{R}^d)$,
\begin{equation}\label{ineq:dispest}
\|P_{s,t}(\omega)\varphi\|_{L^p(\mathbb{R}^d)} \lesssim \frac{1}{|W_t(\omega)-W_s(\omega)|^{d(1/2-1/p)}} \|\varphi\|_{L^{p'}(\mathbb{R}^d)},
\end{equation}
where $p'$ is the Hölder conjugate of $p$.

Let $T>0$, $(q,p)$ sub-admissible and $\omega\notin \mathcal{N}$ where $\mathcal{N}$ is given by Theorem \ref{thm:HLSsto} with $\alpha = d(1/2-1/p)$. We denote $\left(P^{*}_{s,t}\right)_{t,s\geq 0}$ the adjoint of the propagator of the white noise dispersion, that is
\begin{equation*}
P^*_{s,t}\varphi(x) : = \mathcal{F}^{-1}\left(e^{i |\xi|^2(W_t-W_s)} \hat{\varphi}(\xi)\right) = P_{t,s}\varphi(x).
\end{equation*}
This leads, in particular, to the fact that
\begin{equation*}
P^*_{s,t} = P_{t,s},\quad P_{0,s}^*P_{0,t} =  P_{s,t}\quad\textrm{and}\quad P_{s,t}P^*_{r,t} = P_{s,r},\; \forall r\in[s,t].
\end{equation*}
We consider the integral, $\forall f,g\in \mathcal{C}([0,T], \mathcal{C}^{\infty}_0(\mathbb{R}^d))$,
\begin{align*}
I(f,g) = \left| \int_0^T\int_0^T \langle P_{0,t} f(s), P_{0,s} g(t)\rangle_{L^2} ds dt \right| &= \left| \int_0^T\int_0^T \langle P_{0,s}^* f(s), P_{0,t}^*g(t)\rangle_{L^2} ds dt \right|
\\ &= \left| \int_0^T\int_0^T \langle P_{s,t} f(s), g(t)\rangle_{L^2} ds dt \right|
\end{align*}
It follows by Hölder's inequality, \eqref{ineq:dispest} and Theorem \ref{thm:HLSsto}, that  $\forall p\in(1,\infty)$,
\begin{align*}
I(f,g) &\leq  \int_0^T\int_0^T \|P_{s,t}(\omega)f(s)\|_{L^{p}(\mathbb{R}^d)}\|g(t)\|_{L^{p'}(\mathbb{R}^d)} ds dt
\\&\lesssim\int_0^T\int_0^T |W_t(\omega)-W_s(\omega)|^{d(1/2-1/p)}\|f(t)\|_{L^{p'}(\mathbb{R}^d)}\|g(s)\|_{L^{p'}(\mathbb{R}^d)} ds dt
\\&\lesssim \|f\|_{L^{q_1}([0,T],L^{p'}(\mathbb{R}^d))}\|g\|_{L^{q_2}([0,T],L^{p'}(\mathbb{R}^d))},
\end{align*}
where $q_1,q_2\in (1,\infty)$ verify
\begin{equation*}
2 - \frac{d}{2}\left(\frac12-\frac1p\right) > \frac 1{q_1} + \frac1{q_2}.
\end{equation*}
Setting $q_1 = q_2 = q'$, the previous inequality becomes
\begin{equation*}
\frac2q> \frac{d}{2}\left(\frac12-\frac1p\right).
\end{equation*}
This yields, on one hand, that
\begin{equation}\label{eq:StriHomoDual}
\left\|\int_0^T P_{0,s}^*(\omega) f(s) ds \right\|_{L^2(\mathbb{R}^d)}^2 = I(f,f) \lesssim  \|f\|_{L^{q'}([0,T],L^{p'}(\mathbb{R}^d))}^2,
\end{equation}
and, on another hand, by a duality argument,
\begin{equation}\label{eq:Stri2b}
\left\| \int_0^T P_{s,\cdot}(\omega) f(s) ds \right\|_{L^{q}([0,T],L^p(\mathbb{R}^d))} \lesssim\|f\|_{L^{q'}([0,T],L^{p'}(\mathbb{R}^d))}
\end{equation}
We are now in position to prove \eqref{eq:Stri1} and \eqref{eq:Stri2}. It follows from \eqref{eq:StriHomoDual} that, $\forall f\in L^2(\mathbb{R}^d)$ and $\forall g\in L^{q'}([0,T];L^{p'}(\mathbb{R}^d))$,
\begin{align*}
\int_{0}^T \langle P_{0,t}(\omega)f,g(t) \rangle_{L^2} dt &= \left\langle f, \int_0^T P_{0,t}^*(\omega)g(t)\right\rangle_{L^2} \leq \|f\|_{L^2(\mathbb{R}^d)} \left\|\int_0^T P_{0,t}^*(\omega) g(t) ds \right\|_{L^2(\mathbb{R}^d)}^2
\\ &\lesssim \|f\|_{L^2(\mathbb{R}^d)} \|g\|_{L^{q'}([0,T],L^{p'}(\mathbb{R}^d))},
\end{align*}
which leads to \eqref{eq:Stri1} by a duality argument. We now turn to \eqref{eq:Stri2}. We have, by \eqref{eq:StriHomoDual},
\begin{align*}
\left\|\int_0^T P_{s,\cdot}(\omega) f(s) ds\right\|_{L^q([0,T];L^p(\mathbb{R}^d))} &\leq \int_0^T \left\|P_{s,\cdot} f(s) \right\|_{L^q([0,T];L^p(\mathbb{R}^d))} ds
\\ &\leq  \int_0^T \|f(s)\|_{L^2(\mathbb{R}^d))} ds = \|f\|_{L^1([0,T];L^2(\mathbb{R}^d))}.
\end{align*}
Thanks to this estimate and an interpolation argument with \eqref{eq:Stri2b}, we deduce \eqref{eq:Stri2}. This concludes the proof of Proposition \ref{prop:Strichartz}. 

\subsection{Proof of Theorem \ref{thm:main}}

We can now apply the previous result to solve the global Cauchy problem of \eqref{eq:main}. First, we rewrite the equation in its mild formulation, $\forall t\in[0,T]$ and $\forall x\in\mathbb{R}^d$,
\begin{equation}\label{eq:mWNSchrodinger}
\psi(t,x) = P_{0,t}\psi_0(x) - i\lambda \int_0^t P_{s,t}|\psi|^{2\sigma}\psi(s,x) ds,
\end{equation}
and assume that $\sigma< \frac{4} d$.
Let $T>0$ and $(q,p)$ be sub-admissible that we will fix later and we consider equation \eqref{eq:mWNSchrodinger} for all $\omega\notin\mathcal{N}$, where $\mathcal{N}$ is given by Proposition \ref{prop:Strichartz}. We consider the  mapping $\Gamma$ from $L^{q}([0,T];L^p(\mathbb{R}^d))$ to itself given by
\begin{equation}\label{eq:GammaMap}
\Gamma(\psi)(t,x) = P_{0,t}\psi_0(x) - i\lambda \int_0^T P_{s,t}|\psi|^{2\sigma}\psi(s,x) ds.
\end{equation}
We denote $B_{R,L^{q}([0,T];L^p(\mathbb{R}^d))}$ a closed ball of $L^{q}([0,T];L^p(\mathbb{R}^d))$ of radius $R>0$ that will be set later. For any $\psi \in B_{R,L^{q}([0,T];L^p(\mathbb{R}^d))}$, we apply the $L^q([0,T];L^p(\mathbb{R}^d))$ norm to \eqref{eq:GammaMap} and deduce, thanks to \eqref{eq:Stri1} and \eqref{eq:Stri2},
\begin{align*}
\|\Gamma(\psi)\|_{L^{q}([0,T];L^p(\mathbb{R}^d))}\leq C_1\|\psi_0\|_{L^2(\mathbb{R}^d)} + C_2|\lambda| \left\|\psi\right\|_{L^{r'(2\sigma+1)}([0,T];L^{l'(2\sigma+1)} (\mathbb{R}^d))}^{2\sigma+1}.
\end{align*}
for any $(r,l)$ sub-admissible. By choosing $(q,p) = (r,l) = (a,2\sigma+2)$, with $a$ such that 
\begin{equation}\label{ineq:achoice}
\frac{d\sigma}4< \frac{2(\sigma+1)}a< 1,
\end{equation}
we have that $(a,2\sigma+2)$ is sub-admissible and that
\begin{equation*}
l' = \frac{l}{l-1} = \frac{2\sigma+2}{2\sigma +1}.
\end{equation*}
Hence, we obtain, by Hölder's inequality,
\begin{equation*}
\left\|\psi\right\|_{L^{r'(2\sigma+1)}([0,T];L^{l'(2\sigma+1)} (\mathbb{R}^d))}^{2\sigma+1} = \left\|\psi\right\|_{L^{r'(2\sigma+1)}([0,T];L^{2\sigma+2} (\mathbb{R}^d))}^{2\sigma+1} \leq T^{1-\frac{2\sigma +2}a} \left\|\psi\right\|_{L^{a}([0,T];L^{2\sigma+2} (\mathbb{R}^d))}^{2\sigma+1},
\end{equation*}
which gives us
\begin{equation}\label{eq:Gammaest1}
\|\Gamma(\psi)\|_{L^{q}([0,T];L^p(\mathbb{R}^d))}\leq C_1\|\psi_0\|_{L^2(\mathbb{R}^d)} + C_2|\lambda| T^{1-\frac{2\sigma +2}a} \left\|\psi\right\|_{L^{a}([0,T];L^{2\sigma+2} (\mathbb{R}^d))}^{2\sigma+1}.
\end{equation}
By similar computations, we obtain that, $\forall \psi_1,\psi_2\in B_{R,L^{q}([0,T];L^p(\mathbb{R}^d))}$,
\begin{equation}\label{eq:Gammaest2}
\|\Gamma(\psi_1) - \Gamma(\psi_2)\|_{L^{q}([0,T];L^p(\mathbb{R}^d))}\leq C_2|\lambda| T^{1-\frac{2\sigma +2}a} R^{2\sigma}\left\|\psi_1 - \psi_2\right\|_{L^{a}([0,T];L^{2\sigma+2} (\mathbb{R}^d))}.
\end{equation}
We remark that, thanks to \eqref{ineq:achoice}, we have
\begin{equation*}
1 - \frac{2\sigma+2}a > 0.
\end{equation*}
Hence, by setting
\begin{equation*}
R = 2 C_1\|\psi_0\|_{L^2(\mathbb{R}^d)},
\end{equation*}
and taking $T>0$ small enough to have
\begin{equation*}
C_2|\lambda| T^{1-\frac{2\sigma +2}a} R^{2\sigma}< 1,
\end{equation*}
we can see that $\Gamma$ is a contraction from $B_{R,L^{a}([0,T];L^{2\sigma+2}(\mathbb{R}^d))}$ to itself. It follows from a Banach fixed point theorem that there exists a unique local solution to \eqref{eq:mWNSchrodinger} in $B_{R,L^{q}([0,T];L^p(\mathbb{R}^d))}$. Since $R$ and $T$ are independent of this solution, we can iterate this procedure to construct a solution in $B_{R,L^{a}([0,+\infty[;L^{2\sigma+2}(\mathbb{R}^d))}$ which concludes the proof of Theorem \ref{thm:main}.
\bibliographystyle{plain}
\bibliography{Bibliography.bib}

\end{document}